\documentclass[9pt, english, a4paper]{article}
\usepackage[latin1]{inputenc}
\usepackage[T1]{fontenc}
\usepackage{babel, graphicx, textcomp, varioref, amsfonts}
\usepackage{amsthm}
\usepackage{amsmath}

\newtheorem{thm}{Theorem}[section]

\newtheorem{lemma}[thm]{Lemma}
\newtheorem{prop}[thm]{Proposition}
\newtheorem{cor}[thm]{Corollary}

\title{Quasi-linear Stochastic Partial Differential Equations with irregular Coefficients - Malliavin regularity of the Solutions}

\author{ Torstein Nilssen \thanks{Department of Mathematics, University of Oslo, Moltke Moes vei 35, P.O. Box 1053 Blindern, 0316 Oslo, Norway.
E-mail: torsteka@math.uio.no. 
Funded by  Norwegian Research Council (Project 230448/F20).} }

\begin{document}

\maketitle

\begin{abstract}

We study quasi-linear stochastic partial differential equations with discontinuous drift coefficients. Existence and uniqueness of a solution is already known under weaker conditions on the drift, but we are interested in the regularity of the solution in terms of Malliavin calculus. We prove that when the drift is bounded and measurable the solution is directional Malliavin differentiable.

\end{abstract}

\section{Introduction}

We consider the quasi-linear stochastic partial differential equation
\begin{equation} \label{SPDE}
\frac{\partial}{\partial t} u(t,x) = \frac{\partial^2}{\partial x^2} u(t,x) + b(u(t,x)) + \frac{\partial^2}{ \partial t \partial x} W(t,x), \hspace{0.5cm} (t,x) \in (0,T] \times (0,1)
\end{equation}
with the initial condition $u(0,x) = u_0(x)$, $u_0 \in C_0((0,1))$. We will consider Neumann boundary conditions, 
$$
\frac{\partial}{\partial x} u(t,0) = \frac{\partial}{\partial x} u(t,1) = 0 .
$$

In (\ref{SPDE}) $\frac{\partial^2}{ \partial t \partial x} W(t,x)$ denotes space-time White noise and we assume $b : \mathbb{R} \rightarrow \mathbb{R}$ is bounded and measurable, i.e. we allow for for discontinuities.

Existence and uniqueness of a strong solution to (\ref{SPDE}) is already known under weaker conditions on the drift. More specifically, in \cite{GyongyPardoux} the authors prove existence and uniqueness of a strong solution to (\ref{SPDE}) when $b$ is allowed to be of linear growth. 

In this paper we are restricting ourselves to bounded drift, but we show that the solution has regularity properties. Indeed, the solution is Malliavin differentiable in every direction, $h \in L^2([0,T] \times [0,1])$, denoted $D^hu(t,x)$. Although we are not yet able to prove existence of the usual Malliavin derivative, i.e. 
$$
D_{\cdot} u(t,x) \in L^2(\Omega ; L^2([0,T] \times [0,1]))
$$
such that $\int_0^T \int_0^1 D_{\theta, \xi} u(t,x) h(\theta,\xi) d\xi d\theta = D^hu(t,x)$, this paper has some major contributions:

\begin{itemize}

\item 
This work shows that the solution behaves more regular than one could expect by considering the drift. The classical way of studying Malliavin calculus and S(P)DE's is to show that the solutions 'inherit' regularity from the coefficients. In the current paper we show that this technique does not reveal all properties of S(P)DE's.

\item
It is an example of an infinite-dimensional generalization of \cite{PMNPZ}. Here, the authors show that SDE's with bounded and measurable drift has a unique strong solution using a new technique which moreover show that the solution is Malliavin differentiable.

\item
Very recently, the authors of \cite{DaPratoFlandoliPriolaRockner} show that there is strong uniqueness (and thus strong existence) in the Hilbert-space valued SDE
$$
dX_t = ( AX_t + B(t,X_t))dt + dW_t \in H
$$
when $B: [0,T] \times H \rightarrow H$ is bounded and measurable. Thus proving a generalization of the famous result by Veretennikov \cite{Veretennikov} and Zvonkin \cite{Zvonkin} to SPDE's.

The current paper suggest that the technique in \cite{PMNPZ} could be used to show that the solutions obtained in \cite{DaPratoFlandoliPriolaRockner} are even Malliavin differentiable.

See also \cite{FlandoliNilssenProske} where the authors prove Malliavin differentiability in the case of H\"{o}lder-continuous drift.

\item
The Malliavin calculus is tailored to investigate regularity properties of densities of random variables. Perhaps the most well known explicit formula for this is the following: for a random variable $F \in \mathbb{D}^{1,2}$, $h \in H$ such that $\langle DF, h\rangle \neq 0$ and $\frac{h}{\langle DF,h\rangle} \in \textrm{ dom}\delta $ (the Skorohod-integral) the density of $F$ is given by
$$
p_F(x) = E\left[ 1_{(F > x)} \delta \left( \frac{h}{\langle DF,h\rangle} \right) \right].
$$
See \cite{Nualart} Proposition 2.1.1 and Exercise 2.1.3 for details and precise formulations. In the above we note that only the directional Malliavin derivative appears.

\end{itemize}

Let us briefly explain the idea of the proof: assume first that $b \in C^1$ and $u$ solves (\ref{SPDE}). The directional Malliavin derivative should then satisfy, for any direction $h \in L^2([0,T] \times (0,1))$,
$$
\frac{\partial}{\partial t} D^hu(t,x) = \frac{\partial^2}{\partial x^2} D^hu(t,x) + b'(u(t,x))D^hu(t,x) + h(t,x) .
$$ 
For a fixed sample path, we regard the above equation as a deterministic equation and we can use the Feynman-Kac formula to solve it as a functional of $\int_0^tb'(u(s,\cdot)))ds$. Since the solution of (\ref{SPDE}) is very irregular as a function of $t$, the local time $L(t, \cdot)$ is continuously differentiable in the spatial variable. Therefore we can write
\begin{align*}
\left| \int_0^tb'(u(s,x))ds \right| & =\left|  \int_{\mathbb{R}} b'(y) L(t,y) dy \right| \\
&  = \left| -\int_{\mathbb{R}} b(y) L'(t,y) dy \right|  \\
& \leq \|b\|_{\infty} \int_{\mathbb{R}} \left| L'(t,y)\right| dy \\
\end{align*}
where we have used integration by parts. We then obtain a priori bounds of $E[(D^hu(t,x))^2]$ which do not depend on the norm of $b'$, but rather on $\|b\|_{\infty}$. Finally we approximate a general $b$ by smooth functions and use comparison to generate strong convergence (in $L^2(\Omega)$) of the corresponding sequence of solutions to the solution of (\ref{SPDE}).

\bigskip

The paper is organized as follows: 
In Section \ref{MalliavinFramework} we introduce the Malliavin calculus and some results we need. In Section \ref{Framework} we state rigorously the equation (\ref{SPDE}). In Section \ref{LocalTime} we prove that the local time of the solution to (\ref{SPDE}) with $b=0$ has nice regularity properties. We then study (\ref{SPDE}) when the drift is smooth in Section \ref{DerivativeFreeSection} and use the results from Section \ref{LocalTime} to obtain derivative-free estimates.

The main result, Theorem \ref{MainResult}, and its proof is in Section \ref{MainSection}.

\section{Basic concepts of Malliavin Calculus} \label{MalliavinFramework}
Let $(\Omega, \mathcal{F}, P)$ be a complete probability space. We assume that $\mathcal{F}$ is the completion of $ \sigma\{ W(h) : h \in L^2([0,T] \times [0,1]) \}$ with the $P$-null sets. Here  \\ $W: L^2([0,T] \times [0,1]) \rightarrow L^2(\Omega)$ is a linear mapping such that $W(h)$ is a centered Gaussian random variable. 
The covariance is given by $E[W(h)W(g)] = \langle h, g \rangle$ where the right hand side denotes the inner product in $L^2([0,T] \times [0,1])$.

We have the orthogonal Wiener chaos decomposition
$$
L^2(\Omega) = \bigoplus_{n=0}^{\infty} H_n ,
$$
where $H_n := span\{ I_n(f) : f \in L^2(([0,T] \times [0,1])^n)\}$ and $I_n(f)$ is the n-fold Wiener-It\^{o} integral of $f$. For a random variable $F \in L^2(\Omega)$ with Wiener chaos decomposition $F = \sum_{n=0}^{\infty} I_n(f_n)$ we have
$$
E[F^2] = \sum_{n =0}^{\infty} n! \|f_n\|^2_{L^2(([0,T] \times [0,1])^n)} .
$$

We call a random variable $F$ smooth if it is of the form
$$
F = f(W(h_1), \dots , W(h_n))
$$
for $h_1, \dots h_n \in L^2([0,T] \times [0,1])$ and $f \in C^{\infty}_c (\mathbb{R}^n)$. For such a random variable we define the Malliavin derivative
$$
D_{\theta, \xi} F = \sum_{j=1}^n \frac{ \partial}{\partial x_j} f(W(h_1), \dots , W(h_n)) h_j(\theta, \xi)
$$
as an element of $L^2 (\Omega ; L^2([0,T] \times [0,1]))$. We denote by $\mathbb{D}^{1,2}$ the closure of the set of smooth random variables with respect to the norm
$$
\|F\|^2_{1,2} := E[F^2] + \int_0^T \int_0^1 E[( D_{\theta, \xi} F)^2] d\xi d\theta .
$$

Furthermore we define the directional Malliavin derivative in the direction $h \in L^2([0,T] \times [0,1])$ as
$$
D^hF = \langle DF,h \rangle = \int_0^T \int_0^1 D_{\theta, \xi} F h(\theta,\xi) d\xi d\theta
$$
and by $\mathbb{D}^{h,2}$ the closure of the set of smooth random variables with respect to the norm
$$
\|F\|^2_{h,2} := E[F^2] + E[(D^hF)^2].
$$

The integration by parts formula
$$
E[ D^h F] = E[FW(h)]
$$
is well known, and can be found in \cite{Nualart}.

We have the following characterization of $\mathbb{D}^{h,2}$ which is obtained by modifying the proof of Proposition 1.2.1 in \cite{Nualart}:

\begin{prop} \label{ChaosCrit}
For $F = \sum_{n =0}^{\infty} I_n(f_n) \in L^2(\Omega)$ we have that $F$ belongs to $\mathbb{D}^{h,2}$ if and only if 
$$
\sum_{n=1}^{\infty} nn! \| \int_0^1 f_n(\cdot, s,y) h(s,y)dy \|^2_{L^2(([0,T] \times [0,1])^{n-1})} < \infty ,
$$
in which case the above is equal to $E[(D^hF)^2]$.
\end{prop}

Let us prove the following technical result which is inspired by Lemma 1.2.3. in \cite{Nualart}:

\begin{lemma} \label{DerivativeCrit}
Suppose $\{ F_N \}_{N \geq 1}  \subset \mathbb{D}^{h,2}$ is such that
\begin{itemize}

\item
$F_N \rightarrow F$ in $L^2(\Omega)$

\item
$\sup_{N \geq 1} E[(D^hF_N)^2] < \infty$

\end{itemize}

Then $F \in \mathbb{D}^{h,2}$ and $D^hF_{N}$ converges to $D^hF$ in the weak topology of $L^2(\Omega)$.

\end{lemma}

\begin{proof}
We write
$$
F= \sum_{n =0}^{\infty} I_n(f_n)
$$
and
$$
F_N= \sum_{n =0}^{\infty} I_n(f_{n,N} ) .
$$
Since $\{ D^hF_N \}_{N \geq 1}$ is bounded in $L^2(\Omega)$ we may extract a subsequence $D^hF_{N_k}$ converging in the weak topology to some element $\alpha = \sum_{n=0}^{\infty} I_n ( \alpha_n)$.
We note that
$$
D^hF_{N_k} = \sum_{n =1}^{\infty} nI_{n-1}( \langle f_{n,N_k}, h \rangle ) 
$$
and we see that $\langle f_{n,N_k}, h \rangle$ converges weakly in $L^2(([0,T] \times [0,1])^{n-1})$ to $\alpha_n$. It follows that $\alpha_n$ coincides with $\langle f_n, h \rangle$ and we have
\begin{align*}
\sum_{n=1}^{\infty} nn! \| \langle f_n, h \rangle \|^2_{L^2(([0,T] \times [0,1])^{n-1})} & \leq \sup_{k \geq 1} \sum_{n=1}^{\infty} nn! \| \langle f_{n,N_k}, h \rangle \|^2_{L^2(([0,T] \times [0,1])^{n-1})} \\
\end{align*}
which is finite by assumption. From Proposition \ref{ChaosCrit} we have $F \in \mathbb{D}^{h,2}$.

If we take any other weakly converging subsequence of $\{ D^hF_N \}_{N \geq 1}$ its limit must converge, by the preceding argument, to $D^hF$. This implies that the full sequence converges weakly.

\end{proof}

\section{Framework and solutions} \label{Framework}

With the notation from the previous section, we define $W(t,A) := W(1_{[0,t] \times A})$ which is the White noise on $[0,T] \times [0,1]$ and for $h \in L^2([0,T] \times [0,1])$ the Wiener-It\^{o}-integral w.r.t. $dW(t,x)$ is equal to 
$$
\int_0^T \int_0^1 h(t,x) dW(t,x) = W(h).
$$

Throughout this paper we will assume we have a filtration $\{ \mathcal{F}_t \}_{t \in [0,T]}$, where $\mathcal{F}_t$ is generated by $\{ W(s,x) : (s,x) \in [0,t] \times [0,1] \}$ augmented with the set of $P$-null sets.

We denote by $G(t,x,y)$ the fundamental solution to the heat equation, i.e.

$$
\frac{\partial}{\partial t} G(t,x,y) = \frac{\partial^2}{\partial x^2} G(t,x,y), \hspace{0.5cm} (t,x) \in (0,T] \times (0,1)
$$
with boundary conditions $\frac{\partial}{\partial x} G(t,0,y) = \frac{\partial}{\partial x} G(t,1,y) =0$ and $\lim_{t \rightarrow 0} G(t,x,y) = \delta_{x}(y)$ - the Dirac delta distribution in $x$.

It is well known that 
$$
G(t,x,y) = \frac{1}{\sqrt{2\pi t}} \sum_{n \in \mathbb{Z}} \left\{ \exp\{ -\frac{(y-x-2n)^2}{4t} \} + \exp\{ -\frac{(y+x-2n)^2}{4t} \} \right\},
$$

and there exist positive constants $c$ and $C$ such that uniformly in $t' < t$ and $y \in [0,1]$ we have
\begin{equation} \label{SemigroupHolderCont}
c \sqrt{t-t'} \leq \int_{t'}^t \int_0^1 G^2(t-s,x,y)dyds \leq C \sqrt{t-t'} .
\end{equation}

\bigskip

Assume we are given a bounded and measurable function $b: \mathbb{R} \rightarrow \mathbb{R}$.
By a solution to our main SPDE, (\ref{SPDE}), we shall mean an adapted and continuous random field $u(t,x)$ such that
\begin{align} \label{SPDESolution}
u(t,x) & = \int_0^1 G(t,x,y)u_0(y)dy \\
\notag & + \int_0^t \int_0^1 G(t-s,x,y) b(u(s,y)) dy ds + \int_0^t \int_0^1 G(t-s,x,y)dW(s,y)  .
\end{align}

\section{Local Time estimates} \label{LocalTime}

The local time of a process $(X_t)_{t \in [0,T]}$ is defined as follows: we define the occupation measure 
$$
\mu_t(A) = | \{ s \in [0,t] : X_s \in A \}|, \hspace{0.5cm} A \in \mathcal{B}(\mathbb{R})
$$
where $|\cdot|$ denotes the Lebesgue measure. The process $X$ has local time on $[0,t]$ if $\mu_t$ is absolutely continuous w.r.t. Lebesgue measure, and the local time, $L(t, \cdot)$, is defined as the corresponding Radon-Nykodim derivative. I.e.
$$
\mu_t(A) = \int_A L(t,y)dy .
$$
The local time satisfies the occupation time density formula
\begin{equation} \label{occupationFormula}
\int_0^t f(X_s)ds = \int_{\mathbb{R}} f(y) L(t,y) dy , \, \, P-a.s.
\end{equation}
for any bounded and measurable $f: \mathbb{R} \rightarrow \mathbb{R}$.

The aim of this section is to study local times of the driftless stochastic heat equation 
$$
\frac{\partial}{\partial t} u(t,x) = \frac{\partial^2}{\partial x^2} u(t,x) + \frac{ \partial^2}{ \partial t \partial x} W(t,x), \, \, (t,x) \in [0,T] \times (0,1) 
$$
with Neumann boundary conditions. We assume $u_0 = 0$ for simplicity. The solution is given by
$$
u(t,x) = \int_0^t \int_0^1 G(t-s,x,y) dW(y,s),
$$
where $G$ is the fundamental solution of the heat equation.

Fix $ x \in [0,1] $ and let $\omega \in C([0,T]; [-x,1-x])$. We are interested in the stochastic process
$$
X_t = u(t, x + \omega(t)) = \int_0^t \int_0^1 G(t-s,x + \omega(t),y) dW(y,s) .
$$
Notice that we are \emph{not} expanding the dynamics in $t$ of the composition of $u$ and $\omega$. Indeed, $x \mapsto u(t,x)$ is $P$-a.s. not differentiable so it is not clear how such a dynamic evolves. And even worse -  there is no It\^{o} formula for this process.

Nevertheless, $X_t$ is a Gaussian process and we have for $t> t'$
\begin{align*}
E[ (X_t - X_{t'})^2] & = \int_0^T \int_0^1 \left\{ G(t-s,x+\omega(t),y)1_{[0,t]}(s)  \right. \\
& \left. - G(t' -s,x + \omega(t'),y)1_{[0,t']}(s) \right\}^2 dy ds \\
& \geq \int_{t'}^t \int_0^1 G^2(t-s,x + \omega(t),y) dyds \geq c \sqrt{t-t'} \\
\end{align*}

from (\ref{SemigroupHolderCont}).

From \cite{GemanHorowitz}, Theorem 28.1 we have
\begin{thm}
Suppose $X_t$ is a Gaussian process such that 
$$
\int_0^T \int_0^T \left( E[(X_t - X_{t'})^2] \right)^{-p - 1/2} dt dt' < \infty .
$$
Then, there exists a local time $L^X(t,\cdot)$ of $X$ which moreover is $\lfloor p \rfloor$ times differentiable.

\end{thm}

We see that the local time of $X_t$ is in $C^1$. 

Moreover, $X_t$ satisfies the following strong local non-determinism:
\begin{lemma}
For all $t_1 < \dots t_n< t \in [0,1]$ we have
$$
Var(X_t | X_{t_1}, \dots X_{t_n}) \geq c \sqrt{t-t_n}
$$
\end{lemma}

\begin{proof}
The conditional variance of $X_t$ given $X_{t_1}, \dots X_{t_n}$ is the square of the distance between $X_t$ and the subspace $span\{X_{t_1}, \dots X_{t_n}\}$. By distance here, we mean in the Hilbert-space $L^2(\Omega)$. We have that for $\alpha_1, \dots \alpha_n \in \mathbb{R}$
$$
X_t - \sum_{j=1}^n \alpha_j X_{t_j} 
$$
$$
\hspace{1cm}  = \int_0^T \int_0^1 1_{[0,t]} G(t-s,x+\omega(t),y) - \sum_{j=1}^n \alpha_n 1_{[0,t_j]} G(t_j-s,x+\omega(t_j),y) dW(y,s)  . \\
$$
We get that
\begin{align*}
E[ (X_t - \sum_{j=1}^n \alpha_j X_{t_j} )^2] & \geq \int_{t_n}^t \int_0^1 G^2(t-s,x+\omega(t),y) dy ds \\
& \geq c \sqrt{t-t_n} .
\end{align*}

\end{proof}

We have the following estimates on the local time

\begin{lemma} \label{momentEstimate}
There exists a constant $C$, not depending on $m$ such that 
$$
E[ | \partial_y L^X(t,y) |^m] \leq \frac{C^m t^{m/4} m!}{ (\Gamma ( \frac{m}{4} + 1))^{1/3} } .
$$

\end{lemma}

\begin{proof}
We note that it is sufficient to prove that 
$$
E[ |L^X(t,y+h) - L^X(t,y)|^m] \leq \frac{C^m |h|^m t^{m/4} m!}{ (\Gamma ( \frac{m}{4} + 1))^{1/3}}
$$
for all real numbers $h$. To this end, we will use the local time formula 
$$
L^X(t,y) = (2\pi)^{-1}  \int_0^t \int_{\mathbb{R}}\exp\{ iu(X_s - y) \} du ds
$$
i.e. using the Fourier transform method.

We have 

\begin{align*}
|L(t,y+h) - L(t,y)|^m & = (2\pi)^{-m}  |\int_0^t \int_{\mathbb{R}}\exp\{ iu(X_s - y) \} \left( e^{-iuh}-1 \right) du ds|^m \\ 
& = (2 \pi)^{-m} m! \int_{0 < s_1 < \dots s_m < t}  \int_{\mathbb{R}^m } \prod_{j=1}^m \exp\{ iu_j(X_{s_j} - y) \}  \\ 
& \, \, \, \, \, \, \, \, \,  \times \prod_{j=1}^{m} \left( e^{-iu_jh}-1 \right) du_1 \dots du_m   ds_1 \dots ds_m. \\ 
& = (2 \pi)^{-m} m! \int_{0 < s_1 < \dots s_m < t}   \int_{\mathbb{R}^m } \exp\{ i\sum_{j=1}^m v_j(X_{s_j} - X_{s_{j-1}}) \} \\
& \, \, \, \, \, \, \, \, \,  \times \prod_{j=1}^{m} \left( e^{-i(v_j - v_{j+1})h}-1 \right) dv_1 \dots dv_m  ds_1 \dots ds_m. \\ 
\end{align*}

Above we have used the change of variables $u_m = v_m$ and $u_j = v_j - v_{j+1}$, that is $u = Mv$ where 
$$ 
M = 
\left( \begin{array}{rrrrr}
1 & -1 & \dots & 0 & 0 \\
0 & 1 & \dots & 0 & 0 \\
\vdots & \vdots  & \ddots & \vdots &  \vdots\\ 
0 & \dots & \dots & 1 & -1 \\
0 & \dots & \dots & 0 & 1  \\
\end{array}
\right) .
$$

For notational convenience we have used $X_{s_0} = y$ and $v_{m+1} = 0$.

Taking the expectation we get

\begin{align*}
E[ |L(t,y+h) - L(t,y)|^m ] & = (2 \pi)^{-m} m! \int_{0 < s_1 < \dots s_m < t}   \int_{\mathbb{R}^m } | E[\exp\{ i\sum_{j=1}^m v_j(X_{s_j} - X_{s_{j-1}}) \} ]| \\
& \hspace{0.7cm}  \times \prod_{j=1}^{m} \left| e^{-i(v_j - v_{j+1})h}-1 \right| dv_1 \dots dv_m  ds_1 \dots ds_m. \\ 
& = (2 \pi)^{-m} m! \int_{0 < s_1 < \dots s_m < t}  \int_{\mathbb{R}^m } \exp\{ -\frac{1}{2} Var( \sum_{j=1}^m v_j(X_{s_j} - X_{s_{j-1}})) \}  \\
& \hspace{0.7cm}  \times \prod_{j=1}^{m} \left| e^{-i(v_j - v_{j+1})h}-1 \right| dv_1 \dots dv_m  ds_1 \dots ds_m. \\
& \leq (2 \pi)^{-m} m! \int_{0 < s_1 < \dots s_m < t}   \int_{\mathbb{R}^m } \exp\{ -\frac{c}{2} \sum_{j=1}^m v^2_jVar(X_{s_j} - X_{s_{j-1}})) \}  \\
& \hspace{0.7cm}  \times \prod_{j=1}^{m} \left| e^{-i(v_j - v_{j+1})h}-1 \right| dv_1 \dots dv_m  ds_1 \dots ds_m. \\
& \leq (2 \pi)^{-m}|h|^m m! \int_{0 < s_1 < \dots s_m < t}  \int_{\mathbb{R}^m } \exp\{ -\frac{c^2}{2} \sum_{j=1}^m v^2_j \sqrt{s_j - s_{j-1}} \}  \\
& \hspace{0.7cm}  \times \prod_{j=1}^{m} \left| v_j - v_{j+1} \right| dv_1 \dots dv_m  ds_1 \dots ds_m. \\
\end{align*}

where we have used the local non-determinism in the second-to-last inequality, and $Var(X_{s_j} - X_{s_{j-1}}) \geq c \sqrt{s_j - s_{j-1}}$ in the last.

We write
\begin{align*}
\int_{\mathbb{R}^m } \exp\{ -\frac{c^2}{2} \sum_{j=1}^m v^2_j \sqrt{s_j - s_{j-1}} \}  \prod_{j=1}^{m} \left| v_j - v_{j+1} \right| dv_1 \dots dv_m  \\
= (2\pi)^{m/2} |\Sigma|^{1/2} E[ \prod_{j=1}^m |X_j - X_{j+1}| ] \\
\end{align*}

where $X \sim \mathcal{N}(0, \Sigma)$, and we have defined $(\Sigma)_{j,k} = \delta_{j,k} (c^2 \sqrt{s_j - s_{j-1}})^{-1}$. Let $Y = MX$, so that $Y \sim \mathcal{N}(0,M \Sigma M^T)$ and it follows from \cite{LiWei} that 
$$
E[ \prod_{j=1}^m |X_j - X_{j+1}| ] = E[ \prod_{j=1}^m |Y_j| ] \leq \sqrt{ \textrm{per}(M \Sigma M^T)} .
$$
Above, per($A$) denotes the permanent of the matrix $A$. Consequently 
\begin{align*}
\int_{\mathbb{R}^m } \exp\{ -\frac{c^2}{2} \sum_{j=1}^m v^2_j \sqrt{s_j - s_{j-1}} \}  \prod_{j=1}^{m} \left| v_j - v_{j+1} \right| dv_1 \dots dv_m  ds_1 \dots ds_m  \\
\leq (2\pi)^{m/2} \sqrt{|\Sigma|} \sqrt{ \textrm{per}(M \Sigma M^T)}  .\\
\end{align*}

Using H\"{o}lder's inequality we get  
\begin{align*}
\int_{0 < s_1 < \dots < s_m < t}  \sqrt{|\Sigma| \textrm{per}(M \Sigma M^T)} ds & \leq \left(\int_{0 < s_1 < \dots < s_m < t}  |\Sigma|^{p/2} ds \right)^{1/p} \\
& \times \left(\int_{0 < s_1 < \dots < s_m < t}  |\textrm{per}(M \Sigma M^T)|^{q/2} ds \right)^{1/q}. \\
\end{align*}

One can check that there exists a constant $C_1>0$, such that 
\begin{align*}
 \int_{0 < s_1 < \dots < s_m < t}  |\Sigma|^{p/2} ds  & =   c^{-pm} \int_{0 < s_1 < \dots < s_m < t} \prod_{j=1}^m |s_j - s_{j-1}|^{-p/4} ds \\
& \leq \frac{C_1^m t^{(4-p)m/4}}{\Gamma{ (\frac{(4-p)m}{4}+1})} \\
\end{align*}
when $p < 4$. We can find a constant $C_2 > 0$ such that

$$
 \int_{0 < s_1 < \dots < s_m < t}   |\textrm{per}(M \Sigma M^T)|^{q/2} ds  \leq C_2^m
$$
when $q < 2$. The proof is technical and is postponed to the Appendix, Section \ref{appendix}.

This gives 

$$
E[|L(t,y+h)- L(t,y)|^m] \leq \frac{C^m |h|^m t^{m/4} m!}{ (\Gamma ( \frac{(4-p)m}{4} + 1))^{1/p}}
$$

for an appropriate constant $C$, and we choose $p=3$ and $q = 3/2$ to get the result.

\end{proof}

We are ready to conclude this section with its most central result:

\begin{prop} \label{integralMoments}
There exists a constant $C > 0$ such that 
$$
E\left[ \left( \int_{\mathbb{R}} | \partial_y L(t,y) | dy \right )^m \right] \leq \frac{C^m t^{m/4} \sqrt{(2m)!}}{ (\Gamma ( \frac{m}{2} + 1))^{1/6}}
$$
\end{prop}

\begin{proof}
We begin by noting that for any $p \geq 1$ we have $ E[ \sup_{0 \leq t \leq T} |X_t|^p] < \infty$. To see this, note that we may regard $u(t,x)$ as a $C([0,T] \times [0,1])$-valued Gaussian random variable. From \cite{Fernique} we get that $E[\|u\|_{\infty}^p] < \infty $ for all $p \geq 1$, so that 
$$
E[\sup_{0 \leq t \leq T} |X_t|^p] = E[ \sup_{0 \leq t \leq T} |u(t,x + \omega(t))|^p]  \leq E[ \sup_{(t,y) \in [0,T] \times [0,1]} |u(t,y)|^p] < \infty .
$$


We may write
\begin{align*}
E\left[ \left( \int_{\mathbb{R}} | \partial_y L(t,y) | dy \right )^m \right] & = E\left[ \left( \int_{\mathbb{R}} | \partial_y L(t,y) | dy \right )^m \right] \\
& \leq 2^{m-1} E\left[ \left( \int_{|y| < 1} | \partial_y L(t,y) | dy \right )^m \right] \\
& + 2^{m-1} E\left[ \left( \int_{|y|\geq 1} | \partial_y L(t,y) | dy \right )^m \right].
\end{align*}
For the first term we can estimate
\begin{align*}
E\left[ \left( \int_{|y| < 1} | \partial_y L(t,y) | dy \right )^m \right] & \leq  \int_{|y|\leq 1} E[| \partial_y L(t,y) |^m] dy \\
& \leq  \sup_{ y \in \mathbb{R}} \left( E[| \partial_y L(t,y) |^{2m}] \right)^{1/2} .\\
\end{align*}

For the second term, we note that from (\ref{occupationFormula}) that the support of  $L(t,\cdot)$ is included in the interval $[ - X_t^*, X_t^* ]$. This gives
\begin{align*}
E\left[ \left( \int_{|y|\geq 1} 1_{\{|y| \leq X_t^*\}}| \partial_y L(t,y) | dy \right )^m \right] & =   \int_{B} E\left[ \prod_{j=1}^m 1_{\{|y_j| \leq X_t^*\}}| \partial_y L(t,y_j) | \right] dy \\
& \leq  \int_{B} \prod_{j=1}^m \left(E\left[ 1_{\{ |y_j| \leq X_t^*\}}| \partial_y L(t,y_j) |^m \right]\right)^{1/m} dy .
\end{align*}
Above we have denoted $B = \{ y \in \mathbb{R}^m | \, |y_j| \geq 1 \, \forall j \, \}$. We use the estimate
\begin{align*}
E\left[ 1_{\{ |y_j| \leq X_t^*\}}| \partial_y L(t,y_j) |^m \right] & \leq \left( P(|y_j| \leq X_t^* )\right)^{1/2}  \left( E[| \partial_y L(t,y_j) |^{2m}] \right)^{1/2} \\
 & \leq ( E[|X_t^*|^4] )^{1/2}|y_j|^{-2} \sup_{ y \in \mathbb{R}} \left( E[| \partial_y L(t,y) |^{2m}] \right)^{1/2} \\
\end{align*}
where we have used Chebyshevs inequality in the last step. This gives
$$
E\left[ \left( \int_{|y|\geq 1} 1_{\{|y| \leq X_t^*\}}| \partial_y L(t,y) | dy \right )^m \right]  
$$

$$
\leq   \sup_{ y \in \mathbb{R}} \left( E[| \partial_y L(t,y) |^{2m}] \right)^{1/2} ( E[|X_t^*|^4] )^{1/2} \left( \int_{|y| \geq 1} |y|^{-2} dy \right)^m .
$$
The result follows from Lemma \ref{momentEstimate}.

\end{proof}


\section{Derivative free estimates} \label{DerivativeFreeSection}

In this section we assume that $b \in C^1_c (\mathbb{R})$ and denote by $u$ the solution to (\ref{SPDE}).

Since $b$ is continuously differentiable it is well known that $u(t,x)$ is Malliavin differentiable, and we have 
$$
D_{\theta, \xi} u(t,x) = G(t-\theta,x,\xi) + \int_{\theta}^t \int_0^1 G(t-s, x,y)b'(u(s,y)) D_{\theta, \xi}u(s,y) dy ds.
$$

Let now $h \in C^2([0,T] \times [0,1])$. Then the random field 
\begin{align*}
v(t,x) & := \int_0^T \int_0^1D_{\theta, \xi} u(t,x) h(\theta, \xi) d\xi d\theta \\
 & = \int_0^t \int_0^1D_{\theta, \xi} u(t,x) h(\theta, \xi) d\xi d\theta \\
\end{align*}
satisfies the following linear equation 
$$
v(t,x) = \int_0^t \int_0^1 G(t-\theta,x,\xi)h(\theta , \xi)d\xi d\theta + \int_0^t \int_0^1 G(t-s, x,y)b'(u(s,y)) v(s,y) dy ds,
$$ 
or, equivalently
$$
\frac{\partial}{\partial t} v(t,x) = \frac{\partial^2}{\partial x^2} v(t,x) + b'(u(t,x)) v(t,x) + h(t,x) , \hspace{0.5cm} (t,x) \in (0, T] \times (0,1) ,
$$
with initial condition $v(0,x) = 0$ and Neumann boundary conditions.

If we let $\mu_x$ denote the measure on $(C([0,T]), \mathcal{B}(C([0,T]))$ such that $\omega \mapsto \omega(s)$ is a doubly reflected (in 0 and 1) Brownian motion starting in $x$, then we get from the Feynman-Kac formula that the above equation is uniquely solved by
\begin{equation} \label{FeynmanKacFormula}
v(t,x) = \int_{C([0,T])} \int_0^t h(t-r, \omega(r)) \exp\{ \int_0^r b'(u(t-s, \omega(s))) ds \} dr d\mu_x(\omega) .
\end{equation}

\begin{lemma} \label{DerivativeFreeBound}
There exists an increasing continuous function \\ $C: [0,\infty) \rightarrow [0, \infty)$ such that 
$$
E[ v^2(t,x)] \leq C(\|b\|_{\infty} ) \left( \int_{C([0,T])} \int_0^t |h(t-r,\omega(r))| dr d\mu_x(\omega) \right)^2 .
$$

\end{lemma}

\begin{proof}
Define the measure $\tilde{P}$ by 
\begin{align*}
d\tilde{P} & := Z dP\\
Z & := \exp\{ - \int_0^T \int_0^1 b(u(s,y)) dW(y,s) - \frac{1}{2} \int_0^T \int_0^1 b^2(u(s,y)) dyds \}.
\end{align*}
Then $\tilde{P}$ is a probability measure and under $\tilde{P}$,
$$
d\tilde{W}(y,s) := b(u(s,y)) dsdy + dW(y,s)
$$
is space-time white noise. Under this measure we have that $u$ is Gaussian, and more precisely
$$
u(t,x) = \int_0^t \int_0^1 G(t-s,x,y) d\tilde{W}(y,s) .
$$

From (\ref{FeynmanKacFormula}) we double the variables to get
\begin{align*}
E[v^2(t,x)]  & = \int_{C([0,T])} \int_{C([0,T])} \int_0^t \int_0^t h(t-r_1, \omega(r_1)) h(t-r_2, \tilde{\omega}(r_2)) \\
 & \times E\left[  \exp\{ \int_0^{r_1} b'(u(t-s,  \omega(s))) ds \} \right. \\
& \times \left. \exp\{ \int_0^{r_2} b'(u(t-s,  \tilde{\omega}(s))) ds \} \right]dr_1 dr_2 d\mu_x(\omega) d\mu_x( \tilde{\omega}) \\
&  \leq \int_{C([0,T])} \int_{C([0,T])} \int_0^t \int_0^t |h(t-r_1,  \omega(r_1))| |h(t-r_2,  \tilde{\omega}(r_2))| \\
 & \times \left( E\left[  \exp\{ 2\int_0^{r_1} b'(u(t-s, \omega(s))) ds \} \right] \right)^{1/2} \\
& \times \left( E \left[ \exp\{ 2\int_0^{r_2} b'(u(t-s,  \tilde{\omega}(s))) ds \} \right] \right)^{1/2} dr_1 dr_2 d\mu_x(\omega) d\mu_x( \tilde{\omega}) \\
\end{align*}

Now we write 
\begin{align*}
E \left[ \exp\{ 2\int_0^{r} b'(u(t-s, \omega(s))) ds \} \right] & = \tilde{E} \left[ \exp\{ 2\int_0^r b'(u(t-s, \omega(s))) ds \} Z^{-1} \right] \\
& \leq \left( \tilde{E} [ \exp\{ 4\int_0^r b'(u(t-s, \omega(s))) ds \}]\right)^{1/2} \\
&  \hspace{0.5cm} \times \left( \tilde{E}[Z^{-2}] \right)^{1/2} \\
\end{align*}

Denote by $L(r,y)$ the local time of the process  $(u(t-s, \omega(s)))_{s \in [0,r]}$. From the occupation time density formula and integration by parts:
\begin{align*}
\tilde{E} [ \exp\{ 4\int_0^t b'(u(t-s, \omega(s))) ds \}] & = \tilde{E} [ \exp\{ 4 \int_{\mathbb{R}} b'(y) L(r,y) dy \}] \\
& = \tilde{E} [ \exp\{ -4\int_{\mathbb{R}} b(y) \partial_y L(r,y) dy \}] .\\
& \leq \tilde{E} [ \exp\{ 4 \|b\|_{\infty} \int_{\mathbb{R}} |\partial_y L(r,y)|  dy \}] \\
\end{align*}

From Proposition \ref{integralMoments} we have 
\begin{align*}
\tilde{E} [ \exp\{ 4 \|b\|_{\infty} \int_{\mathbb{R}} |\partial_y L(r,y)|dy \}] & = \sum_{m \geq 0} \frac{(4 \|b\|_{\infty})^m}{m!} \tilde{E} \left[\left(\int_{\mathbb{R}} |\partial_y L(r,y)|dy \right)^m \right]  \\
& \leq \sum_{m \geq 0} \frac{(4 \|b\|_{\infty})^m  C^m \sqrt{(2m)!}}{ m! (\Gamma(\frac{m}{2}+1))^{1/6}}   \\
& =: \tilde{C}( \|b\|_{\infty}) \\
\end{align*}

which converges by Stirling's formula. 

It is easy to see that we can bound $\tilde{E}[Z^{-2}]$ by a function only depending on $\|b\|_{\infty}$. 

Combining the above we get
\begin{align*}
E[v^2(t,x)] & \leq C(\|b\|_{\infty}) \int_{C([0,T])}\int_{C([0,T])} \int_0^t \int_0^t  |h(t-r_1, \omega(r_1))| \\
& \hspace{1.5cm} \times |h(t-r_2,  \tilde{\omega}(r_2))| dr_1 dr_2 d\mu_x(\omega) d\mu_x( \tilde{\omega})  \\
& = C(\|b\|_{\infty}) \left( \int_{C([0,T])} \int_0^t  |h(t-r,  \omega(r))|  dr d\mu_x(\omega) \right)^2 \\
\end{align*}
for an appropriate function $C$, and the result follows.
\end{proof}

In the above we assumed that $h \in C^2([0,T] \times [0,1])$. We may extend this to $h \in L^2([0,T] \times [0,1])$.

\begin{cor}
For any $h \in L^2([0,T] \times [0,1])$ we have
$$
E[ (D^h u(t,x))^2]  \leq C(\|b\|_{\infty} ) \sqrt{t} \|h\|^2_{L^2([0,T] \times [0,1])}
$$

\end{cor}

\begin{proof}
We know that the random variable $\omega \mapsto \omega(r)$ has density $G(r,x,\cdot)$ under $\mu_x$.
From Lemma \ref{DerivativeFreeBound} we see that for $h \in C^2([0,T] \times [0,1])$, by H\"{o}lder's inequality
\begin{align*}
E[(D^h u(t,x))^2] & \leq C( \|b\|_{\infty}) \left( \int_0^t \int_0^1 |h(t-r,y)| G(r,x,y) dy dr \right)^2 \\
& \leq  C( \|b\|_{\infty})   \int_0^t \int_0^1 |h(r,y)|^2  dy dr  \int_0^t \int_{\mathbb{R}} G^2(r,x,y) dy dr \\
& \leq C( \|b\|_{\infty})  \|h\|^2_{L^2([0,T] \times [0,1])} C \sqrt{t} . \\
\end{align*}

Consequently we may extend the linear operator
\begin{align*}
L^2([0,T] \times [0,1]) & \rightarrow  L^2(\Omega) \\
h  & \mapsto  D^h u(t,x) \\
\end{align*}
by continuity. The result follows.

\end{proof}

\section{Directional Derivatives when the drift is discontinuous} \label{MainSection}

In \cite{GyongyPardoux} the authors successfully generalize the famous results by Zvonkin \cite{Zvonkin} and Veretennikov \cite{Veretennikov} to infinite dimension, i.e. they show that (\ref{SPDE}) has a unique strong solution when $b$ is bounded and measurable. In fact, they show that this holds true even when the drift is of linear growth.

Let us briefly explain the idea of the proof; let $b$ be bounded and measurable and define for $n \in \mathbb{N}$
$$
b_n(x) := n\int_{\mathbb{R}} \rho(n(x-y)) b(y) dy 
$$
where $\rho$ is a non-negative smooth function with compact support in $\mathbb{R}$ such that $\int_{\mathbb{R}} \rho(y)dy = 1$.

We let 
$$
\tilde{b}_{n,k} := \bigwedge_{j=n}^k b_j, \hspace{1cm} n \leq k
$$
and 
$$
B_n = \bigwedge_{j=n}^{\infty}  b_j,
$$
so that $\tilde{b}_{n,k}$ is Lipschitz. Denote by $\tilde{u}_{n,k}(t,x)$ the unique solution to (\ref{SPDE}) when we replace $b$ by $\tilde{b}_{n,k}$. Then one can use comparison to show that 
$$
\lim_{k \rightarrow \infty} u_{n,k}(t,x) = u_n(t,x),  \hspace{1cm} \textrm{in } L^2(\Omega)
$$ 
where $u_n(t,x)$ solves (\ref{SPDE}) when we replace $b$ by $B_n$. Furthermore, 
$$
\lim_{n \rightarrow \infty} u_{n}(t,x) = u(t,x),  \hspace{1cm} \textrm{in } L^2(\Omega)
$$ 
where $u(t,x)$ is a solution to (\ref{SPDE}). For details see \cite{GyongyPardoux}.


We are ready to prove our main theorem:

\begin{thm} \label{MainResult}
Assume $b$ is bounded and measurable. Denote by $u$ the solution of (\ref{SPDE}). Then for every $h \in L^2([0,T] \times [0,1])$ we have 
$$
u(t,x) \in \mathbb{D}^{h,2} .
$$
\end{thm}

\begin{proof}
From the discussion above we know that we have $u_n(t,x) \rightarrow u(t,x)$ in $L^2(\Omega)$. From Lemma \ref{DerivativeFreeBound} we see that 
$$
\sup_{n \geq 1} E[ (D^h u_n(t,x))^2] < \infty
$$
for any $h \in L^2([0,T] \times [0,1])$. It follows from Lemma \ref{DerivativeCrit} that $u(t,x) \in \mathbb{D}^{h,2}$.
\end{proof}

\section{Appendix} \label{appendix}
Consider the matrices from Section \ref{LocalTime},
$$
\Sigma = 
\left( \begin{array}{cccc}
(s_1 - s_2)^{-1/2} & 0 & \dots & 0 \\
0 & (s_2 - s_3)^{-1/2} & \dots & 0 \\
\vdots & \vdots  & \ddots &  \vdots\\ 
0 & \dots & (s_{m-1} - s_m)^{-1/2} & 0 \\
0 & \dots & 0 & s_m^{-1/2}  \\
\end{array}
\right)
$$
and 
$$
M = 
\left( \begin{array}{rrrrr}
1 & -1 & \dots & 0 & 0 \\
0 & 1 & \dots & 0 & 0 \\
\vdots & \vdots  & \ddots & \vdots &  \vdots\\ 
0 & \dots & \dots & 1 & -1 \\
0 & \dots & \dots & 0 & 1  \\
\end{array}
\right) .
$$

The purpose of this section in to show that the function $f_m(s_1, \dots s_m) := \textrm{per}( M \Sigma M^T)$ is such that for $\beta \in (0,1)$ we have
$$
\int_{0 < s_m < \dots s_1 < t} |f_m(s_1, \dots, s_m)|^{\beta} ds_m \dots ds_1 \leq C^m
$$
for some constant $C = C(t,\beta)$.

We start by noting that 
$$
M\Sigma M^T  = 
\left( \begin{array}{rrrrrr}
a_1 & b_1 & 0 & \dots & \dots & 0 \\
b_1 & a_2 & b_2&  \dots & \dots & 0 \\
0 & b_2 & a_3 &  \dots & \dots & 0 \\
\vdots & \vdots &   & \ddots  & \dots &  \vdots\\ 
0 & \dots & \dots& 0 & a_{m-1} & b_{m-1} \\
0 & \dots & \dots& 0 & b_{m-1} & a_m  \\
\end{array} 
\right) 
$$

where 
$$
a_j = 
\left\{ 
\begin{array}{ll}
(s_j - s_{j+1})^{-1/2} +(s_{j+1} - s_{j+2})^{-1/2} &  \textrm{ for } j=1, \dots m-2 \\
(s_{m-1} - s_m)^{-1/2} + s_m^{-1/2}  &  \textrm{ for } j= m-1\\
s_m^{-1/2} &  \textrm{ for } j= m \\
\end{array} \right.
$$

and $b_j = -(s_{j+1} - s_{j+2})^{-1/2}$ for $j = 1, \dots, m-2$.

Using the definition of the permanent of a matrix we see that we have the following recursive relation

\begin{align*}
f_m(s_1, \dots , s_m)   =  & \left( (s_1 - s_2)^{-1/2} +(s_2 - s_3)^{-1/2} \right) f_{m-1}( s_2, \dots ,s_m) \\
& + (s_2- s_3)^{-1} f_{m-2}(s_3, \dots ,s_m) \\
\end{align*}
with
$$
f_1(s_1) = s_1^{-1/2} \, \, \textrm{ and  } \, \, f_2(s_1, s_2) = (s_2 - s_1)^{-1/2}s_2^{-1/2} + s_2^{-1} .
$$

We write $f_m(s_1,\dots ,s_m) = p_m((s_1 - s_2)^{-1/2}, \dots ,(s_{m-1} - s_m)^{-1/2}, s_m^{-1/2})$ where $p_m$ is the polynomial recursively defined by
$$
p_m(x_1, \dots, x_m)   =   ( x_1 + x_2 ) p_{m-1}( x_2, \dots , x_m)  + x_2^2 p_{m-2}(x_3 \dots ,x_m) 
$$
with
$$
p_1(x_1) = x_1 \textrm{ and  } \, \, p_2(x_1, x_2) = x_1x_2 + x_2^2 .
$$

If we denote by $deg_{x_i}p_m$ the degree of the polynomial in the variable $x_i$, for $i=1, \dots m$ we see from the recursive relation that 
$$
deg_{x_1}p_m = 1 \, \textrm{ and } \, deg_{x_j} p_m \leq 2, \textrm{ for } j=2, \dots , m .
$$


Moreover, if we denote by $\gamma_m$ the number of terms in this polynomial, it is clear from the recursive relation that 
$$
\gamma_m = 2 \gamma_{m-1} + \gamma_{m-2} 
$$
and 
$$
\gamma_1 = 1 \, \, \textrm{ and } \, \, \gamma_2 = 2 .
$$
So that we have $\gamma_m \leq C^m$ for $C$ large enough.

It follows that we may write 
$$
p_m(x_1, \dots, x_m) = \sum_{\alpha} c_{\alpha} x^{\alpha}
$$
where the sum is taken over all multiindices $\alpha \in \mathbb{N}^m$ with $\alpha_i \leq 2$ and $\alpha_1 \leq 1$. Here we have denoted $x^{\alpha} = x_1^{\alpha_1} \dots x_m^{\alpha_m}$. Moreover, there are at most $C^m$ terms in this sum with $C$ as above and one can show that $|c_{\alpha}| \leq 3^m$ for all $\alpha$.

Consequently
$$
|f(s_1, \dots , s_m)|^{\beta} \leq 3^m \sum_{\alpha} |s_1 - s_2|^{-\beta \alpha_1/2} \cdots |s_{m-1} - s_m|^{-\beta \alpha_{m-1}/2} s_m^{-\beta \alpha_m/2} .
$$
Since $\frac{\beta \alpha_i}{2} < 1$ for all $i = 1, \dots , m$, each of the above terms are integrable over $ 0 < s_m < \dots < s_1 < t$, and there are at most $C^m$ such terms. The result follows.

\end{document}